\documentclass[12pt]{amsart}
\usepackage{amsmath,amssymb,mathrsfs,epsfig,mathtools,rotating,hyperref}

\theoremstyle{theorem}
\newtheorem{theo}{Theorem}[section]
\newtheorem{lemm}[theo]{Lemma}
\newtheorem{conj}[theo]{Conjecture}

\theoremstyle{remark}
\newtheorem{rema}[theo]{Remark}

\theoremstyle{definition}
\newtheorem{defi}[theo]{Definition}

\numberwithin{equation}{section}
\numberwithin{figure}{section}
\def\slbf#1{\text{\boldmath$#1$}}

\graphicspath{{./pictures/}}	

\author{Ivan Dynnikov}
\title{On Zeeman's collapsibility conjecture for non-standard polyhedra}
\address{\noindent Steklov Mathematical Institute of Russian Academy of Sciences, 8 Gubkina Str., Moscow 119991, Russia}
\email{dynnikov@mi-ras.ru}

\begin{document}

\maketitle

\begin{abstract}
We prove that if Zeeman's collapsibility conjecture holds for standard polyhedra,
then it holds in general.
\end{abstract}



\section{Introduction}

In 1964, E.\,C.\,Zeeman proposed the following conjecture.
\begin{conj}[\cite{{z64}}]\label{zcc-conj}
Let~$K$ be a contractible compact two-dimensional polyhedron. Then~$K\times\left[0;1\right]$ is collapsible.
\end{conj}

It is well known that this conjecture (which we abbreviate as ZCC)
implies
the Andrews--Curtis conjecture with stabilizations (ACCS), which states the following.

\begin{conj}\label{zccs-conj}
Let~$\langle a_1,\ldots,a_n\,|\,r_1,\ldots,r_n\rangle$ be a balanced presentation
of the trivial group. Then it can be transformed into a trivial balanced presentation
by a sequence of operations of the following kinds\emph:
\begin{enumerate}
\item
the relators~$r_i$ are permuted\emph;
\item
$r_1$ is replaced by~$r_1^{-1}$\emph;
\item
$r_1$ is replaced by~$r_1r_2$\emph;
\item
$r_1$ is replaced by~$wr_1w^{-1}$, where~$w$ is a word in the generators~$a_1^{\pm1},\ldots,a_n^{\pm1}$\emph;
\item
\emph(stabilization\emph)
a new generator~$a_{n+1}$ and a new relator~$r_{n+1}=a_{n+1}$ are added, and $n$ is incremented by~$1$.
\end{enumerate}
\end{conj}

By a \emph{trivial balanced presentation} we mean a group presentation of the form
$$\langle a_1,\ldots,a_m\,|\,a_1,\ldots,a_m\rangle.$$

ACCS is sometimes confused with the original Andrews--Curtis conjecture~\cite{AC65},
which differs by not allowing stabilizations. (However, the version with stabilizations is mentioned in a referee's remark
quoted in~\cite{AC65}.) This confusion is likely due to the fact that ACCS admits
a very short reformulation in topological terms, which reads as follows.

\begin{conj}[Reformulation of ACCS]
Any contractible compact two-polyhedron is three-deformable to a point.
\end{conj}

The Poincar\'e conjecture (PC) is also an quick corollary of ZCC, as was already
noted by Zeeman. In the 1980s, it was discovered that the inverse implications hold
in the case of standard polyhedra: PC and ACCS together imply ZCC restricted to
the class of standard polyhedra.

More precisely, PC implies ZCC restricted to standard spines of compact three-manifolds, as shown
by~D.\,Gillman and D.\,Rolfsen in~\cite{GR83}, whereas ACCS implies ZCC for standard polyhedra
that are not spines of three-manifolds, which is  a result due to S.\,Matveev~\cite{mat87}.

The aim of the present paper is to show that `standard' can be omitted from the aforementioned results
by Gillman--Rolfsen and Matveev. Thus, since the Poincar\'e conjecture has been proven by G.\,Perelman~\cite{p1,p2,p3,mt},
we establish the equivalence of ACCS and ZCC.

While ZCC remains unsettled, one can ask a more general question about the smallest~$r\in\mathbb N$
such that any contractible compact two-polyhedron~$K$ is~$r$-collapsible (meaning that~$K\times\left[0;1\right]^r$
is collapsible). As shown by M.\,Cohen~\cite{cohen} the latter holds for~$r=6$.
We show that the condition of being standard is not essential in this more general context, either.
Namely, we prove the following result.

\begin{theo}\label{main-th}
Suppose that there exists a compact contractible two-polyhedron~$K$ that is not $r$-collapsible
for some~$r\in\mathbb N$. Then there exists a standard contractible two-polyhedron~$K'$
which is not $r$-collapsible.
\end{theo}

This theorem follows from Theorem~\ref{K->semi-standard-th}, which is given in the next section.
The Poincar\'e--Perelman theorem, together with the Gillman--Rolfsen theorem and Theorem~\ref{K->semi-standard-th} below,
also implies the following.

\begin{theo}\label{spines-th}
ZCC holds for polyhedra embeddable into three-manifolds.
\end{theo}

\section{Preliminaries. Standard two-polyhedra}\label{gen-set-st-poly-subsec}

We work in the PL category, but often prefer to use the language of cell complexes. All cell complexes involved are assumed 
to be compact and to carry
a PL structure compatible with their cell structure. All maps involved are assumed to be continuous and piecewise linear
unless otherwise specified.

By `a complex' we mean `a polyhedral cell complex with a fixed cell structure,' whereas `polyhedra' are viewed up
to subdivision. Thus, `a subcomplex' refers to a complex consisting
of whole cells of another complex, whereas `a subpolyhedron' may be an arbitrary subspace that becomes
a simplicial subspace after a suitable subdivision. By `a cell' we mean `an open cell'. Cell complexes are not necessarily
regular; that is, we do not require the attaching
map of each cell to be injective. In particular, the closure of a cell need not be a disc.

When~$C$ is a complex, we denote its~$i$-skeleton by~$C^{(i)}$, where~$i\in\mathbb N\cup\{0\}$.
The unit interval~$\left[0;1\right]$ is denoted by~$I$. It is regarded as a cell complex consisting of three cells:
$\{0\}$, $\{1\}$, and~$(0;1)$.

We denote by~$\Theta_n$, $n\in\mathbb N$ a graph consisting of two vertices and~$n$ edges
connecting them.
The symbol~$\mathscr K_n$, $n\in\mathbb N$,
denotes a complete graph with~$n$ vertices.
By~$\mathscr S_n$, $n\in\mathbb N$, we denote a star graph with~$n$ leaves,
that is, the cone over an $n$-point discrete space.

By~$\mathscr K_{3,3}$ we denote a complete bipartite graph with~$3+3$ vertices. By a \emph{special tree}
we mean a tree in which every vertex has degree either one or three. A disjoint union of special trees is called a \emph{special forest}.
The degree-one vertices of a graph~$G$ are called the \emph{leaves} of~$G$.

\begin{defi}
A finite two-dimensional cell complex~$K$ is called \emph{standard} if the link of
any point~$p\in K^{(1)}\setminus K^{(0)}$ is homeomorphic to~$\Theta_3$,
and the link of any point~$p\in K^{(0)}$ is homeomorphic to~$\mathscr K_4$.
A \emph{standard polyhedron} is the underlying polyhedron of a standard complex.
\end{defi}

For any polyhedron~$K$, we denote by~$K^{(i)\mathrm t}$, $i\in\mathbb N\cup\{0\}$,
the set of points~$p\in K$ whose regular neighborhood does not have the form~$B\times I^{i+1}$,
where~$B$ is a polyhedron. Equivalently, $K^{(i)\mathrm t}$ is the set of points~$p\in K$
that belong to the $i$-skeleton of any cell decomposition of~$K$.
This subset is called the \emph{true $i$-skeleton} of~$K$.
The points in~$K^{(0)\mathrm t}$ are called \emph{true vertices}, and the
connected components of~$K^{(1)\mathrm t}\setminus K^{(0)\mathrm t}$
that are homeomorphic to an open interval are called \emph{true edges} of~$K$.
By the \emph{degree}
of a true edge~$e$ of~$K$ we mean the number~$k$ such that~$\mathrm{lk}(p,K)=\Theta_k$
for~$p\in e$.

For a polyhedron~$K$, we define the \emph{boundary of~$K$}, denoted by~$\partial K$, as
the set of points~$p$ having a regular neighborhood~$A$
such that~$(A,p)$ is homeomorphic to~$(B\times I,q\times1)$, where~$B$ is a polyhedron and~$q\in B$.
(Note that the boundary of a polyhedron is not necessarily a closed subset.)
The boundary (frontier) of a subset~$X$ of a topological space will be denoted by~$\dot X$ in order to distinguish
it from~$\partial X$ when~$X$ is also a subpolyhedron.

\begin{defi}
We say that a complex~$Y$ is obtained from a complex~$X$ by an \emph{elementary collapse}
if~$Y=X\setminus(\sigma\cup\eta)$, where~$\sigma$ and~$\eta$ are two cells of~$X$ such that~$\dim\sigma=\dim\eta+1$
and~$\eta\subset\partial X\cap\dot\sigma$. The cell~$\eta$ is then said to be a \emph{free face} of~$\sigma$
(with respect to~$X$).
We also say that~$Y$ is obtained from~$X$ by \emph{collapsing~$\eta$ across~$\sigma$}
and write~$X\stackrel{\sigma,\eta}\searrow Y$ to indicate this.
(Our definition is slightly more general than that of an elementary simplicial collapse,
but it is well known that the collapsibility relation generated by elementary
collapses is insensitive to this
difference. There are two other commonly used definitions that work equally well.)

A complex is called \emph{collapsible} if it collapses to a single-point space. A polyhedron without
a fixed cell structure is called collapsible if it admits a collapsible cell structure.
\end{defi}

\begin{defi}
We say that two polyhedra~$K$ and~$K'$ are $r$-\emph{deformable} to each other
if there exists a finite sequence of polyhedra~$K_0=K,K_1,K_2,\ldots,K_m=K'$
such that~$\dim K_i\leqslant r$, and either~$K_{i-1}\searrow K_i$
or~$K_i\searrow K_{i-1}$ for all~$i=1,\ldots,m$.
\end{defi}

\begin{defi}
We say that a polyhedron~$K$ is a \emph{connected sum} of polyhedra~$K_1$ and~$K_2$,
and write~$K=K_1\#K_2$,
if there exist points~$p_1\in K_1$ and~$p_2\in K_2$ such that~$K\cong(K_1\sqcup K_2)/(p_1\sim p_2)$.
A connected sum of~$m$ polyhedra with~$m>2$ is defined inductively.

A connected sum~$K_1\#K_2\#\ldots\#K_m$ of several standard polyhedra is called \emph{semi-standard}.
\end{defi}

Theorems~\ref{main-th} and~\ref{spines-th} are corollaries of the following more general statement.

\begin{theo}\label{K->semi-standard-th}
Let~$K$ be a compact contractible two-polyhedron. Then there exists a semi-standard two-polyhedron~$K'=K_1\#K_2\#\ldots\#K_m$ such that
\begin{enumerate}
\item
$K$ is three-deformable to~$K'$\emph;
\item
if~$K$ embeds into a three-manifold, then so does~$K'$\emph;
\item
if~$r\in\mathbb N$ and each~$K_i$ is $r$-collapsible, then~$K$ is also $r$-collapsible.
\end{enumerate}
\end{theo}

The remainder of the paper is devoted to the proof of this theorem.

\section{Ways to be non-standard}

\begin{lemm}\label{cases-k-lem}
Let~$K$ be a contractible compact two-polyhedron. Then at least one of the following statements
holds\emph:
\begin{itemize}
\item[(K1)]
$K$ is standard\emph;
\item[(K2)]
$\partial K$ is not empty\emph;
\item[(K3)]
$K$ has a separating vertex\emph;
\item[(K4)]
there exists a true edge of~$K$ of degree at least four\emph;
\item[(K5)]
there exists a vertex~$v$ of~$K$ such that the graph~$\mathrm{lk}(v,K)$
contains two disjoint cycles\emph;
\item[(K6)]
there exists a vertex~$v$ of~$K$ such that~$\mathrm{lk}(v,K)$ is
homeomorphic to~$\mathscr K_{3,3}$\emph;
\item[(K7)]
there exists a connected component of~$K\setminus K^{(1)\mathrm t}$ that is not simply connected.
\end{itemize}
\end{lemm}

\begin{proof}
Suppose that~(K7) does not hold.
Then a cell decomposition of~$K$ can be chosen so that~$K^{(1)}=K^{(1)\mathrm t}$.
In this case, since~$K$ is contractible, a connected component of~$K^{(1)}$ cannot be a circle unless~$K\cong\mathbb D^2$.
Indeed, otherwise there would be either two discs in~$K$ with boundaries attached to the same circle, or a single disc
attached by a map of degree different from~$\pm1$. In both cases, this would imply~$H_2(K;\mathbb Z)\not\cong\{0\}$,
which contradicts the contractibility of~$K$.

Thus, if neither~(K2) nor~(K7) holds, then, by choosing a cell decomposition of~$K$ appropriately,
we may assume that~$K^{(1)}=K^{(1)\mathrm t}$ and~$K^{(0)}=K^{(0)\mathrm t}$.
Suppose that, endowed with this cell decomposition, $K$ is not standard.
Then either~$K$ has an edge of degree different from three, or~$K$ has a vertex~$v$
such that~$\mathrm{lk}(v,K)$ is not homeomorphic to~$\mathscr K_4$.
In the former case, Condition~(K4) holds.

Suppose that none of the conditions~(K1)--(K4) and~(K7) holds. This implies that
there exists a vertex~$v\in K$ such that~$\mathrm{lk}(v,K)$ is a connected regular three-valent graph~$G$
different from~$\mathscr K_4$.

Since a finite regular three-valent graph cannot be a tree, $G$ must contain cycles. Let~$s$ be a cycle in~$G$
consisting of the smallest possible number of edges, and let~$m$ be this number.

If the closure~$\overline{G\setminus s}$ is not a forest, then Condition~(K5) holds, and we are done.
Suppose otherwise, that each connected component of~$\overline{G\setminus s}$ is a tree.
Then $\overline{G\setminus s}$ is a special forest.
The number of leaves in~$\overline{G\setminus s}$ is~$m$,
since they are precisely the vertices of~$G$ contained in~$s$.

Since a special tree cannot have exactly one leaf,
we must have~$m\geqslant2$. A special tree has exactly two leaves if and only if it consists of a single edge.
In this case, we would have~$m=2$ and~$G\cong\Theta_3$, which contradicts the assumption that~$v$ is a true vertex.
Therefore, $m\geqslant3$.

In any special tree, there are two leaves connected by a path consisting of at most two edges. Therefore,
there are two distinct vertices in~$s$ that are connected by such a path in~$\overline{G\setminus s}$.
Since, by assumption, $s$ is a shortest
cycle in~$G$, we must have~$m\leqslant 4$.

If~$m=3$, the only available option for~$G$ such that~$\overline{G\setminus s}$ is a special forest
is~$G\cong\mathscr K_4$, which is assumed not to be the case.

In the case where~$m=4$, there are two options: the graph~$\mathscr K_{3,3}$ and
the $1$-skeleton of a triangular prism. In the former case, Condition~(K6) holds, whereas in the latter case,
Condition~(K5) holds.\end{proof}

\section{Transformations of two-polyhedra}

Formally, by a \emph{transformation} of polyhedra we mean a binary relation between polyhedra
where one polyhedron is obtained from another through a specific local change.

\begin{defi}\label{admis-defi}
We call a transformation of polyhedra of dimension at most two \emph{$r$-admissible}, where~$r\in\mathbb N$, if whenever~$K'$ is obtained from~$K$
by this transformation, the following statements hold:
\begin{enumerate}
\item
$K$ and~$K'$ are three-deformable to each other;
\item
if~$K$ is embeddable into a three-manifold, then so is~$K'$;
\item
if~$K'$ is $r$-collapsible, then~$K$ is $r$-collapsible, too.
\end{enumerate}
If a transformation is $r$-admissible for every~$r\in\mathbb N$, it is called \emph{admissible}.
\end{defi}

An example of an admissible transformation is an elementary collapse, which clearly satisfies the properties listed above.
We now introduce two more admissible transformations, namely contractions and special extractions. Together
with elementary collapses, these transformations allow one to transform any contractible two-polyhedron into a semi-standard one, as
will be shown in Section~\ref{standardtn-sec}.

\subsection{Contractions}
For a PL map~$\varphi$, we denote by~$\mathrm{Cyl}_\varphi$ the mapping cylinder of~$\varphi$.
When~$L$ is a subpolyhedron of a compact polyhedron~$K$, we denote by~$\mathscr N(L)$
a regular neighborhood of~$L$ in~$K$ and assume that it is endowed with
the mapping cylinder structure of a projection~$\pi_L:\partial\mathscr N(L)\rightarrow L$,
so that~$\mathscr N(L)\cong\mathrm{Cyl}_{\pi_L}$.
When~$Y$ is a subcomplex of a complex~$X$, by~$\mathscr N(Y)$
we mean a regular neighborhood of~$Y$ in~$X$ such that the intersection of~$\mathscr N(Y)$ with each cell~$\sigma$ of~$X$
not contained entirely in~$Y$ is of the form~$(\partial\mathscr N(Y)\cap\sigma)\times\left[0;1\right)$
with respect to the mapping cylinder structure of~$\mathscr N(Y)$.
This can always be achieved by choosing~$\mathscr N(Y)$ sufficiently small and choosing~$\pi_L$ appropriately.

\begin{defi}\label{contract-defi}
Let~$K$ be a compact polyhedron, and let~$L$ be a subpolyhedron of~$K$. Let also~$f:L\rightarrow L'$
be a surjective PL map from~$L$ to another polyhedron.
Then the polyhedron
\begin{equation}\label{Kcontracted-eq}
K'=\overline{K\setminus\mathscr N(L)}\cup_{\partial\mathscr N(L)}\mathrm{Cyl}_{f\circ\pi_L}
\end{equation}
is said to be obtained from~$K$ by \emph{contracting~$L$ to~$L'$} (\emph{via~$f$}).

If, additionally, $L'$ is a single point and~$L$ is collapsible (or, more specifically, $L\cong I^k$ for some~$k\in\mathbb N$),
then we say that~$K'$ is obtained from~$K$ by a \emph{contraction} (respectively, \emph{disc contraction}).
The inverse passage, from~$K'$ to~$K$, is then called an \emph{extraction}.
\end{defi}

Topologically, the polyhedron~\eqref{Kcontracted-eq} can be identified with the quotient space~$K/{\sim}_f$,
where the equivalence relation~$\sim_f$ is as follows:
$$x\sim_f x'\quad\text{if and only if}\quad x=x'\quad\text{or}\quad x,x'\in L\text{ and }f(x)=f(x')$$
(with the sole caveat that the natural projection~$K\rightarrow K/{\sim}_f$ is \emph{not} a PL map).

\begin{lemm}\label{I-contract-collapse-lem}
Let~$K$, $L$, $L'$, and~$f$ be as in Definition~\ref{contract-defi}, and let~$K'$ be obtained from~$K$
by contracting~$L$ to~$L'$ via~$f$. Suppose that~$f:L\rightarrow L'$ is a trivial $I$-fibration
and~$K'$ is collapsible. Then~$K$ is also collapsible.
\end{lemm}

\begin{proof}
We endow~$K'$ with an appropriate collapsible cell structure
such that~$L'$ is a union of cells, and view it as a complex.
Accordingly, we endow~$K$ with a cell decomposition such that, for each cell~$\sigma$ of~$K'$,
the preimage~$p^{-1}(\sigma)$ is a cell of~$K$ provided that~$\sigma\cap L'=\varnothing$,
and a union of cells of~$K$ otherwise, where~$p:K\rightarrow K'=K/{\sim}_f$ is the natural projection.

Let
$$K'=K_0'\stackrel{\sigma_1,\eta_1}\searrow K_1'\stackrel{\sigma_2,\eta_2}\searrow
K_2'\stackrel{\sigma_3,\eta_3}\searrow\ldots\stackrel{\sigma_m,\eta_m}\searrow K_m'=\{\mathrm{pt}\}$$
be a sequence of elementary collapses. Define~$K_i\subset K$ to be~$p^{-1}(K_i')$ with the cell decomposition
inherited from~$K$.

If~$\sigma_i\cap L'=\eta_i\cap L'=\varnothing$, then~$K_i$ is obtained from~$K_{i-1}$ by an elementary collapse.
If~$\sigma_i,\eta_i\subset L'$, then~$K_{i-1}$ collapses to~$K_i$, since one can view each of~$p^{-1}(\sigma_i)$
and~$p^{-1}(\eta_i)$ in~$K_{i-1}$ as a single cell, which makes the transition from~$K_{i-1}$ to~$K_i$
an elementary collapse.

Now suppose that~$\eta_i\subset L'$ and~$\sigma_i\cap L'=\varnothing$. The intersection~$\eta=\overline{p^{-1}(\sigma_i)}\cap p^{-1}(\eta_i)$ meets every fiber of the $I$-fibration~$p^{-1}(\eta_i)\rightarrow\eta_i$
in a connected subset, and~$\eta$ is topologically a cell.
Therefore, $\overline{p^{-1}(\eta_i)}$ collapses to~$p^{-1}(\dot\eta_i)\cup\eta$, and~$\overline{p^{-1}(\sigma_i)}$
collapses to~$p^{-1}(\dot\sigma_i\setminus\eta_i)$. This implies that~$K_{i-1}\searrow K_i\cup\overline{p^{-1}(\sigma_i)}\searrow K_i$.

Thus, in all cases, $K_{i-1}$ collapses to~$K_i$, which means that~$K$ is collapsible.
\end{proof}

The following two lemmas demonstrate that a contraction is an admissible transformation.

\begin{lemm}
Let~$K'$ be a polyhedrn obtained from a compact polyhedron~$K$ by a contraction. Suppose that~$K'$ is $r$-collapsible for some~$r\in\mathbb N$. Then~$K$ is
also $r$-collapsible.
\end{lemm}

\begin{proof}
For a disc contraction, this statement follows from
Lemma~\ref{I-contract-collapse-lem} by induction on the dimension of the contracted disc.
An arbitrary contraction can be decomposed into a sequence of disc contractions,
which yields the general case.
\end{proof}

\begin{lemm}
Let~$K'$ be a polyhedron obtained from a compact polyhedron~$K$ by a contraction. Then the following statements hold\emph:

\noindent\emph{(i)}
$K$ and~$K'$ are three-deformable to each other\emph;

\noindent\emph{(ii)}
if~$K$ embeds into a three-manifold, then so does~$K'$.
\end{lemm}

\begin{proof}
(i) Denote by~$L$ the collapsible subpolyhedron in~$K$ that contracts to a point to yield~$K'$.
Let~$K''$ be the polyhedron obtained from~$K$ by attaching a cone over~$\mathscr N(L)$.
Then~$K''$ contains both~$K$ and~$K'$ and collapses to each of them.

(ii) Suppose that~$K$ is embedded into a three-manifold~$M$. By adding a collar to~$M$,
we may assume that~$\partial M\cap K=\varnothing$. Then~$K'$ is embedded
into the three-manifold obtained from~$M$ by contracting~$L$ to a point.
\end{proof}

\subsection{Special extractions}

An extraction, which is the inverse of a contraction, is not in general an admissible transformation, since
it can transform a spine of a three-manifold into a polyhedron that is not embeddable into
a three-manifold. However, there is a specific kind of extraction that is admissible. We define it below.

\begin{rema}
Property~(2) in Definition~\ref{admis-defi} is important only for establishing Theorem~\ref{spines-th}.
To prove Theorem~\ref{main-th}, one can use any transformations that satisfy only
Properties~(1) and~(3). However, the author is not aware of
any type of transformation for which one can prove~(3) without Property~(2) holding true.
\end{rema}

\begin{defi}\label{resol-def}
Let~$K$ be a two-dimensional polyhedron, and
let~$v\in K\setminus\partial K$. (The construction works equally well when~$v\in\partial K$,
but the definition in that case requires a more complicated system of notation.)
We denote by~$V$ a regular neighborhood~$\mathscr N(v)$ of~$v$,
and by~$G$ the link of~$v$ in~$K$. We have~$\partial V\cong G$.

Choose a decomposition of~$G$ into a union of two subpolyhedra~$G_1$ and~$G_2$
such that
\begin{equation}\label{dG1=dG2-eq}
G_1\cap G_2=\partial G_1=\partial G_2
\end{equation}
and~$G_1\cap G_2$ is contained in a single connected component of~$G_i$ for $i=1,2$.
In~$G_1$ and~$G_2$,
choose contractible subpolyhedra (that is, trees) $T_1$ and~$T_2$, respectively,
such that~$\partial T_1=\partial T_2=\partial G_1$. Observe that Condition~\eqref{dG1=dG2-eq}
implies that the points in~$G_1\cap G_2$ are \emph{not} true vertices of~$G$.

Take the disjoint union of the cones over~$G_1\cup T_2$, $T_1\cup T_2$, and $T_1\cup G_2$,
and glue the first two along~$T_2$ and the last two along~$T_1$. Let~$M$ be the
resulting polyhedron. We can see that~$M$ is collapsible and~$\partial M=G_1\cup G_2=G$.

We denote by~$\mathfrak s(K,T_1,T_2)$ the polyhedron obtained from~$K$ by removing~$V\setminus\partial V$
and gluing in~$M$ along~$\partial M\cong\partial V$ instead. (Observe that if~$G$ is connected, then~$G_1$ and~$G_2$ are uniquely determined by~$T_1$
and~$T_2$; for this reason we omit them from the notation.) We say that the passage from~$K$ to~$\mathfrak s(K,T_1,T_2)$
is a \emph{$T_1,T_2$-extraction at~$v$}.
\end{defi}

One can see that the contraction of~$M$ to a point turns~$\mathfrak s(K,T_1,T_2)$ into a polyhedron
homeomorphic to~$K$. Thus, a $T_1,T_2$-extraction is indeed an extraction. We call extractions of this
form \emph{special} when~$T_1$ and~$T_2$ are not specified.

The following two lemmas demonstrate that a special extraction is an admissible transformation.

\begin{lemm}
If a compact two-polyhedron~$K$ embeds into a three-manifold and~$K'$ is obtained from~$K$
by a special extraction, then~$K'$ also embeds into
a three-manifold.
\end{lemm}

\begin{proof}
We use the notation from Definition~\ref{resol-def}. Suppose first that~$\mathrm{lk}(v,K)$ is connected.

Let~$P$ be a three-manifold containing~$K$. By adding a collar to~$P$, we may assume that~$v\notin\partial P$.
Let~$B\cong I^3$ be a regular neighborhood of~$v$ in~$P$
such that~$B\cap K=V$. We identify~$G$ with~$\partial B\cap K$. The two-sphere~$\partial B$ can be
decomposed into the union of two ``hemispheres'' (two-discs) $H_1$ and~$H_2$ such that~$H_i\cap G=G_i$ ($i=1,2$) and
$H_1\cap H_2=\partial H_1=\partial H_2$. Take two-discs~$D_1,D_2\subset B$ such that~$D_1\cap D_2=D_i\cap\partial B=
\partial H_i$. These discs cut~$B$ into three parts, say~$B_1$, $B_2$, and~$B_3$. By reindexing the discs~$D_i$ and~$B_i$
we may ensure that~$\partial B_1=H_1\cup D_2$, $\partial B_2=D_1\cup D_2$, and~$\partial B_3=D_1\cup H_2$.
For each~$i=1,2$, let~$h_i:H_i\rightarrow D_i$ be a homeomorphism that is the identity on~$\partial H_i$.

Then the cones~$C_1$, $C_2$, and~$C_3$ over~$G_1\cup h_2(T_2)$, $h_2(T_2)\cup h_1(T_1)$, and~$h_1(T_1)\cup G_2$,
respectively, can be embedded into~$B_1$, $B_2$, and~$B_3$ so that~$C_i\cap\partial B_i=\partial C_i$.
The polyhedron~$(K\setminus V)\cup C_1\cup C_2\cup C_3\subset P$ is then homeomorphic to
the polyhedron obtained from~$K$ by the $T_1,T_2$-extraction at~$v$.

To treat the general case, we observe that if a polyhedron~$K_1$ is obtained from a polyhedron~$K_2$
by identifying finitely many points, then~$K_1$ is embeddable into a three-manifold if and only
if so is~$K_2$.

If the link~$\mathrm{lk}(v,K)$ is disconnected, then there exist a polyhedron~$\widetilde K$ and
finitely many points~$v_1,\ldots,v_m\in\widetilde K$, where~$m>1$, such that~$K$
is obtained from~$\widetilde K$ by identifying all these points with each other, with their image being~$v$. Moreover,
we may ensure that the link~$\mathrm{lk}(v_i,\widetilde K)$ is connected for all~$i=1,\ldots,m$. Then
there exists a special extraction~$\widetilde K\mapsto\widetilde K'$ at one of the points~$v_i$
such that~$K'$ can be obtained from~$\widetilde K'$ by a special extraction. This reduces
the general case to the particular one considered above.
\end{proof}

\begin{lemm}\label{resol-collaps-lem}
In the setting of Definition~\ref{resol-def}, the polyhedron~$K\times I$
collapses to a subpolyhedron homeomorphic to~$\mathfrak s(K,T_1,T_2)\times I$.
\end{lemm}

\begin{proof}
We continue to use the notation from Definition~\ref{resol-def} and redefine~$M$ as a subpolyhedron of~$V\times I$.
To this end, we introduce a ``cylindrical coordinate system'' $(\rho,\varphi,z)$ on~$V\times I$, in which
a point~$\varphi\in G$ plays the role of the angular coordinate, and the boundary~$\partial V\cong G$
is defined by~$\rho=1$. In other words,
we view~$V\times I$ as~$(I\times G\times I)/{\sim}$, where
\begin{equation}\label{cyl-ident-eq}
(0,\varphi',z)\sim(0,\varphi'',z)
\end{equation}
for all~$\varphi',\varphi''\in G$.
We denote the set~$\partial G_1=\partial G_2=\partial T_1=\partial T_2$ by~$J$.

We define~$M$ to be the union of the following three subpolyhedra~$M_1$, $M_2$, and~$M_3$:
$$\begin{aligned}
M_1&=\left(1\times G_1\times\left[\frac14;\frac12\right]\right)\cup\left(I\times G_1\times\frac14\right)
\cup\left(\left[0;\frac12\right]\times T_2\times\frac14\right)/{\sim},\\
M_2&=\left(\left[\frac12;1\right]\times J\times\left[\frac14;\frac34\right]\right)\cup
\left(\frac12\times T_1\times\left[\frac12;\frac34\right]\right)\\&\hskip26mm\cup
\left(\left[0;\frac12\right]\times(T_1\cup T_2)\times\frac12\right)\cup
\left(\frac12\times T_2\times\left[\frac14;\frac12\right]\right)/{\sim},\\
M_3&=\left(1\times G_2\times\left[\frac12;\frac34\right]\right)\cup\left(I\times G_2\times\frac34\right)
\cup\left(\left[0;\frac12\right]\times T_1\times\frac34\right)/{\sim}.
\end{aligned}$$

One can verify that each~$M_i$ for~$i=1,2,3$ is a cone over~$\partial M_i$, and
$$\partial M_1=G_1'\cup T_2',\quad\partial M_2=T_1'\cup T_2',\quad\partial M_3=T_1'\cup G_2',$$
where
$$\begin{aligned}
G_1'&=\left(1\times G_1\times\frac12\right)\cup
\left(1\times J\times\left[\frac14;\frac12\right]\right)\cong G_1,\\
G_2'&=\left(1\times G_2\times\frac12\right)\cup
\left(1\times J\times\left[\frac12;\frac34\right]\right)\cong G_2,\\
T_1'&=\left(\frac12\times T_1\times\frac34\right)\cup\left(\left[\frac12;1\right]\times J\times\frac34\right)\cong T_1,\\
T_2'&=\left(\frac12\times T_2\times\frac14\right)\cup\left(\left[\frac12;1\right]\times J\times\frac14\right)\cong T_2.
\end{aligned}$$
One can also see that the subpolyhedra~$M_i$ have no intersections with one another except those mentioned above,
as well as no intersections
with~$K\setminus V$. Furthermore, we have
$$\partial M=G_1'\cup G_2'=\partial\left(\overline{K\setminus V}\times\frac12\right),$$
which implies that the subpolyhedron~$\mathfrak s=\bigl((K\setminus V)\times1/2\bigr)\cup M$ is homeomorphic to~$\mathfrak s(K,T_1,T_2)$.

We now proceed to show that the polyhedron~$K\times I$ collapses to a subpolyhedron
homeomorphic to~$\mathfrak s\times I$. We endow~$K$ with a cell
decomposition and view it as a complex. The polyhedron~$G$ is identified with~$\partial V$.
We also assume that the intersection of~$V$ with~$K^{(1)}$ is the cone
over the vertices of~$G$, and that the points in~$J$ lie outside of~$K^{(1)}$.

For each~$i=1,2$, let~$\widetilde T_i$ be a regular neighborhood of~$T_i$ in~$G_i$,
and let~$J_i=\partial\widetilde T_i\setminus\partial T_i=\partial\,\overline{G_i\setminus\widetilde T_i}$.
We define~$\widetilde M$ similarly to~$M$ using~$\widetilde T_i$ instead of~$T_i$, and
let~$\widetilde{\mathfrak s}$ be the following subpolyhedron of~$K\times I$:
$$\widetilde{\mathfrak s}=\left((K\setminus V)\times\frac12\right)\cup\widetilde M\cup\left(\left(\left[0;\frac12\right]\times
J_1\times\left[\frac12;\frac34\right]\right)
\cup\left(\left[0;\frac12\right]\times J_2\times\left[\frac14;\frac12\right]\right)/{\sim}\right).$$
By definition, the subpolyhedron~$\widetilde{\mathfrak s}$ contains~$\mathfrak s$. We now show
that~$\widetilde{\mathfrak s}$ ``cuts~$K\times I$ into two halves.''

First, observe that, for each $k$-cell~$\sigma$ of the product cell decomposition of~$K\times I$ with~$k=1,2,3$,
each connected component of the intersection~$\widetilde{\mathfrak s}\cap\sigma$ is a $(k-1)$-cell,
and~$\widetilde{\mathfrak s}\cap(K\times I)^{(0)}=\varnothing$. Indeed, all the cells of~$K\times I$
meeting~$\widetilde{\mathfrak s}$ are of the form~$\eta\times I$,
where~$\eta$ is a cell of~$K$. It suffices to examine the intersection~$\widetilde{\mathfrak s}\cap(\overline\eta\times I)$
near~$V\times I$ for a two-cell~$\eta$, since far from~$V$ this intersection coincides with~$\overline\eta\times1/2$.

The closure~$P$ of each connected component of the intersection~$\eta\cap V$, where~$\eta$ is a two-cell of~$K$,
is a cone over an edge~$e$ of~$G$. There are four mutually exclusive cases to consider.

\emph{Case 1}: $e\subset G_i\setminus\widetilde T_i$, where~$i\in\{1,2\}$. By symmetry, it suffices to consider
the subcase~$i=1$. The polyhedron~$\widetilde{\mathfrak s}$ meets~$P\times I$ in
the subpolyhedron
\begin{equation}\label{discG1-eq}
\left(I\times e\times\frac14\right)\cup\left(1\times e\times\left[\frac14;\frac12\right]\right)/{\sim},
\end{equation}
which is a two-disc whose interior is contained in~$\eta\times\left(0;1\right)$, and whose boundary is
$$\left(I\times\partial e\times\frac14\right)\cup\left(1\times\partial e\times\left[\frac14;\frac12\right]\right)
\cup\left(1\times e\times\frac12\right)/{\sim}\subset\left(K^{(1)}\times I\right)\cup\left(\partial V\times\frac12\right).$$

\emph{Case 2}: $e\subset\widetilde T_i$. Again, we may assume that~$i=1$.
The intersection $\widetilde{\mathfrak s}\cap(P\times I)$ has two connected components,
one of which is~\eqref{discG1-eq}, while the other is
$$\left(\left[0;\frac12\right]\times e\times\left\{\frac12,\frac34\right\}\right)\cup
\left(\frac12\times e\times\left[\frac12;\frac34\right]\right)/{\sim}.$$
The latter is also a two-disc whose interior is contained in~$\eta\times\left(0;1\right)$, and whose boundary is
$$\left(\left[0;\frac12\right]\times\partial e\times\left\{\frac12,\frac34\right\}\right)\cup
\left(\frac12\times\partial e\times\left[\frac12;\frac34\right]\right)/{\sim}\subset K^{(1)}\times I.$$

\emph{Case 3}: $e\cap J_i\ne\varnothing$. As before, we consider the subcase~$i=1$.
We have~$e\subset G_1$ and~$e\not\subset\widetilde T_1$.
Let~$e'=e\cap\widetilde T_1$, $e\cap J_1=\{p\}$, and~$\partial e\cap\widetilde T_1=\{q\}$.
The intersection $\widetilde{\mathfrak s}\cap(P\times I)$ has two connected components
one of which is~\eqref{discG1-eq}, while the other is
$$\left(\left[0;\frac12\right]\times e'\times\left\{\frac12,\frac34\right\}\right)\cup
\left(\frac12\times e'\times\left[\frac12;\frac34\right]\right)\cup
\left(\left[0;\frac12\right]\times p\times\left[\frac12;\frac34\right]\right)/{\sim}.$$
The latter is also a two-disc whose interior is contained in~$\eta\times\left(0;1\right)$, and whose boundary is
$$\left(\left[0;\frac12\right]\times q\times\left\{\frac12,\frac34\right\}\right)\cup
\left(\frac12\times q\times\left[\frac12;\frac34\right]\right)/{\sim}\subset K^{(1)}\times I.$$

\emph{Case 4}:
$e\cap J\ne\varnothing$. Let~$e_i=e\cap G_i$, $e\cap J=\{p\}$, and~$\partial e\cap G_i=\{q_i\}$ for~$i=1,2$.
Note that we also have~$e_i\subset\widetilde T_i$.
The intersection $\widetilde{\mathfrak s}\cap(P\times I)$ is
\begin{multline*}\left(1\times e_1\times\left[\frac14;\frac12\right]\right)\cup\left(I\times e_1\times\frac14\right)\cup
\left(\left[\frac12;1\right]\times p\times\left[\frac14;\frac34\right]\right)\cup\\
\left(1\times e_2\times\left[\frac12;\frac34\right]\right)\cup\left(I\times e_2\times\frac34\right)\cup
\left(\left[0;\frac12\right]\times e_1\times\left\{\frac12,\frac34\right\}\right)\cup\\
\left(\frac12\times e_1\times\left[\frac12;\frac34\right]\right)\cup
\left(\left[0;\frac12\right]\times e_2\times\left\{\frac14,\frac12\right\}\right)\cup
\left(\frac12\times e_2\times\left[\frac14;\frac12\right]\right)/{\sim},
\end{multline*}
which is a two-disc whose interior is contained in~$\eta\times\left(0;1\right)$, and whose boundary is
\begin{multline*}
\left(1\times e\times\frac12\right)\cup\left(1\times q_1\times\left[\frac14;\frac12\right]\right)\cup
\left(I\times q_1\times\frac14\right)\cup\left(\left[0;\frac12\right]\times q_2\times\frac14\right)\cup\\
\left(\frac12\times q_2\times\left[\frac14;\frac12\right]\right)\cup
\left(\left[0;\frac12\right]\times q_2\times\frac12\right)\cup\left(\left[0;\frac12\right]\times q_1\times\frac12\right)\cup
\left(\frac12\times q_1\times\left[\frac12;\frac34\right]\right)\cup\\
\left(\left[0;\frac12\right]\times q_1\times\frac34\right)\cup
\left(I\times q_2\times\frac34\right)\cup\left(1\times q_2\times\left[\frac12;\frac34\right]\right)
\subset\left(K^{(1)}\times I\right)\cup\left(\partial V\times\frac12\right)/{\sim}.
\end{multline*}

Thus, in each of these cases,
the intersection~$\widetilde{\mathfrak s}\cap(\eta\times I)$ is a two-disc,
and the intersection~$\widetilde{\mathfrak s}\cap(\dot\eta\times I)$ consists
of open arcs contained in two-cells of~$K\times I$ and finitely many points contained in one-cells.

Thus, $\widetilde{\mathfrak s}$ is of codimension one in every cell of~$K\times I$.
Let~$X$ be the cell complex obtained from~$K\times I$ by a subdivision with the mimimum
possible number of cells such that~$\widetilde{\mathfrak s}$ is a subcomplex. We define
two subsets~$X_-$ and~$X_+$ of~$X$ as follows:
$$\begin{aligned}
X_-&=\left(\overline{K\setminus V}\times\left[0;\frac12\right]\right)\cup\left(\left(\left[\frac12;1\right]\times G_2\times\left[0;\frac34\right]\right)
\cup\left(\left[0;\frac12\right]\times(\widetilde T_1\cup G_2)\times\left[\frac12;\frac34\right]\right)\right.\\
&\hskip-1em\cup\left.\left(I\times G_1\times\left[0;\frac14\right]\right)\cup\left(\left[0;\frac12\right]\times G_2\times\left[0;\frac14\right]\right)
\cup\left(\left[0;\frac12\right]\times(\overline{G_2\setminus\widetilde T_2})\times\left[\frac14;\frac12\right]\right)\right)/{\sim},\\
X_+&=\left(\overline{K\setminus V}\times\left[\frac12;1\right]\right)\cup\left(\left(\left[\frac12;1\right]\times G_1\times\left[\frac14;1\right]\right)
\cup\left(\left[0;\frac12\right]\times(\widetilde T_2\cup G_1)\times\left[\frac14;\frac12\right]\right)\right.\\
&\hskip-1em\cup\left.\left(I\times G_2\times\left[\frac34;1\right]\right)\cup\left(\left[0;\frac12\right]\times G_1\times\left[\frac34;1\right]\right)
\cup\left(\left[0;\frac12\right]\times(\overline{G_1\setminus\widetilde T_1})\times\left[\frac12;\frac34\right]\right)\right)/{\sim}.
\end{aligned}$$

The diagram in Figure~\ref{blocks-fig} shows which product blocks constituting~$X_-$ have two-dimensional
intersections with other blocks, and which parts of their frontiers are ``free'', that is, contained in~$\partial X_-$. (The figure does not show all
triple intersections of the blocks, which are one-dimensional. The part~$K\times 0$ of~$\partial X_-$ is also not shown.)
\begin{figure}[ht]
\includegraphics[width=480pt]{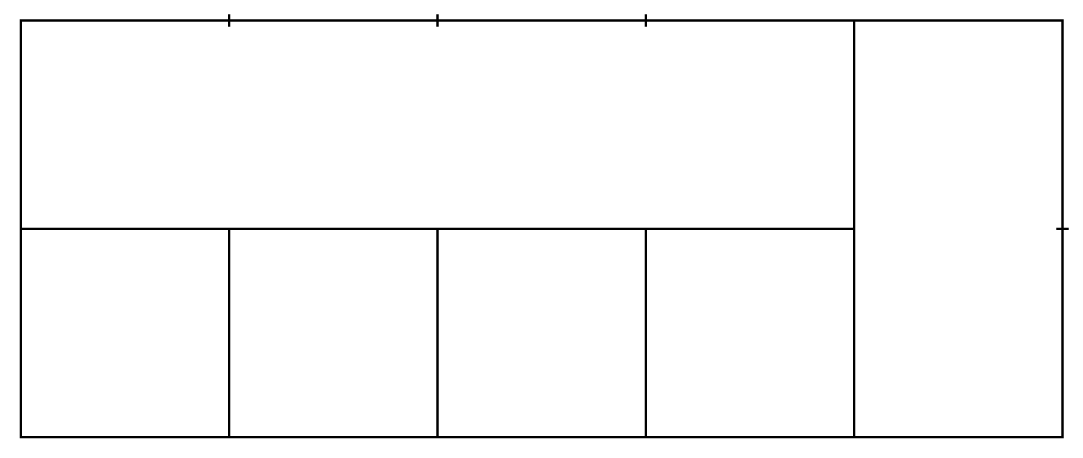}
\put(-445,55){$\scriptstyle\overline{K\setminus V}\times\left[0;\frac12\right]$}
\put(-358,55){$\scriptstyle I\times G_1\times\left[0;\frac14\right]$}
\put(-272,55){$\scriptstyle\left[0;\frac12\right]\times G_2\times\left[0;\frac14\right]$}
\put(-190,55){$\scriptstyle\left[0;\frac12\right]\times(G_2\setminus\widetilde T_2)\times\left[\frac14;\frac12\right]$}
\put(-320,147){$\scriptstyle\left[\frac12;1\right]\times G_2\times\left[0;\frac34\right]$}
\put(-98,100){$\scriptstyle\left[0;\frac12\right]\times(\widetilde T_1\cup G_2)\times\left[\frac12;\frac34\right]$}
\put(-446,108){$\scriptscriptstyle1\times G_2\times\left[0;\frac12\right]$}
\put(-358,108){$\scriptscriptstyle\left[\frac12;1\right]\times J\times\left[0;\frac14\right]$}
\put(-262,108){$\scriptscriptstyle\frac12\times G_2\times\left[0;\frac14\right]$}
\put(-181,108){$\scriptscriptstyle\frac12\times(G_2\setminus\widetilde T_2)\times\left[\frac14;\frac12\right]$}
\put(-482,40){\rotatebox{90}{$\scriptscriptstyle\overline{K\setminus V}\times\frac12$}}
\put(-393,35){\rotatebox{90}{$\scriptscriptstyle1\times G_1\times\left[0;\frac14\right]$}}
\put(-301,30){\rotatebox{90}{$\scriptscriptstyle\left[0;\frac12\right]\times J\times\left[0;\frac14\right]$}}
\put(-208,25){\rotatebox{90}{$\scriptscriptstyle\left[0;\frac12\right]\times(G_2\setminus\widetilde T_2)\times\frac14$}}
\put(-99,88){\rotatebox{-90}{$\scriptscriptstyle\left[0;\frac12\right]\times(G_2\setminus\widetilde T_2)\times\frac12$}}
\put(-99,173){\rotatebox{-90}{$\scriptscriptstyle\frac12\times G_2\times\left[\frac12;\frac34\right]$}}
\put(-447,-1){$\scriptscriptstyle1\times G_1\times\left[\frac14;\frac12\right]$}
\put(-348,-1){$\scriptscriptstyle I\times G_1\times\frac14$}
\put(-260,-1){$\scriptscriptstyle\left[0;\frac12\right]\times\widetilde T_2\times\frac14$}
\put(-176,-1){$\scriptscriptstyle\left[0;\frac12\right]\times J_2\times\left[\frac14;\frac12\right]$}
\put(-448,200){$\scriptscriptstyle\frac12\times\widetilde T_2\times\left[\frac14;\frac12\right]$}
\put(-355,200){$\scriptscriptstyle1\times G_2\times\left[\frac12;\frac34\right]$}
\put(-267,200){$\scriptscriptstyle\left[\frac12;1\right]\times J\times\left[\frac14;\frac34\right]$}
\put(-170,200){$\scriptscriptstyle\left[\frac12;1\right]\times G_2\times\frac34$}
\put(-80,200){$\scriptscriptstyle\frac12\times\widetilde T_1\times\left[\frac12;\frac34\right]$}
\put(-7,178){\rotatebox{-90}{$\scriptscriptstyle\left[0;\frac12\right]\times J_1\times\left[\frac12;\frac34\right]$}}
\put(-7,88){\rotatebox{-90}{$\scriptscriptstyle\left[0;\frac12\right]\times(\widetilde T_1\cup\widetilde T_2)\times\frac12$}}
\put(-87,-1){$\scriptscriptstyle\left[0;\frac12\right]\times(\widetilde T_1\cup G_2)\times\frac34$}
\caption{The structure of~$X_-$ (not to scale). Quotienting by the equivalence~$\sim$ is omitted from the notation}\label{blocks-fig}
\end{figure}
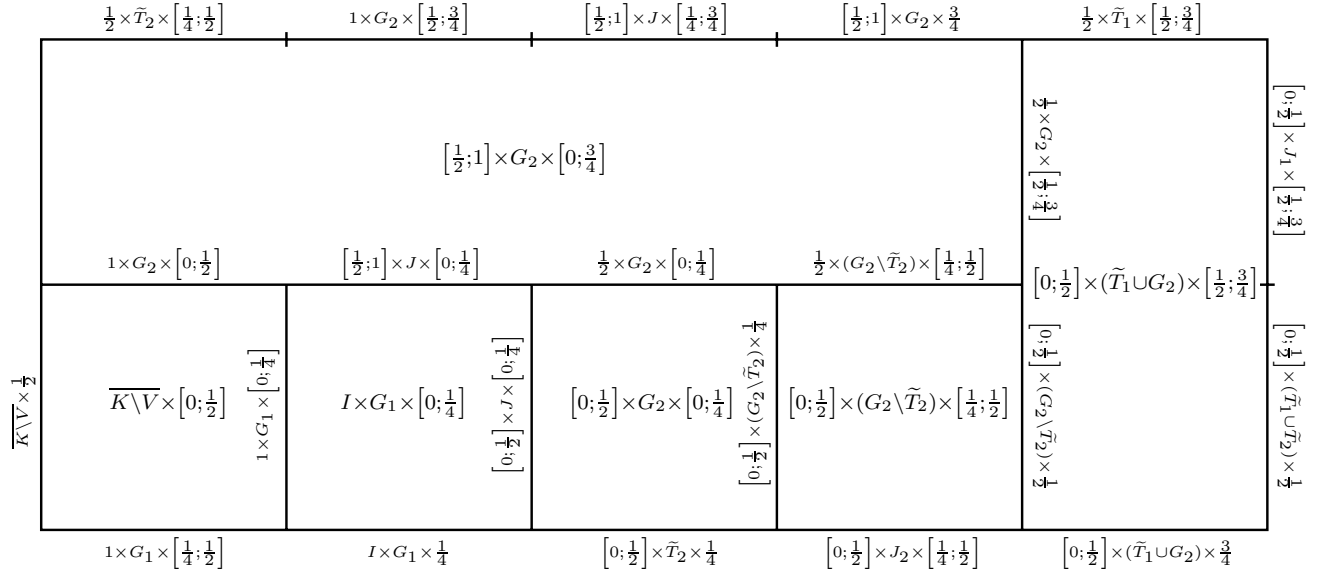
Using this diagram, one can verify that~$\dot X_-=\widetilde{\mathfrak s}$.
By symmetry, we also have~$\dot X_+=\widetilde{\mathfrak s}$.
This implies that~$\widetilde{\mathfrak s}$ is not only of codimension one in every cell, but also two-sided;
that is, its regular neighborhood~$\mathscr N(\widetilde{\mathfrak s})$ is homeomorphic to~$\widetilde{\mathfrak s}\times I$.
We can choose this neighborhood in such a way that any connected component of
the intersection of~$\mathscr N(\widetilde{\mathfrak s})$ with every $k$-cell~$\sigma$ of~$K\times I$ ($k=1,2,3$)
is homeomorphic to~$(\widetilde{\mathfrak s}\cap\sigma)\times I$, which will be assumed in the sequel.
We denote by~$U$ the part of~$\mathscr N(\widetilde{\mathfrak s})$ whose intersection with every cell~$\sigma$ of~$K\times I$
consists only of those components of~$\mathscr N(\widetilde{\mathfrak s})\cap\sigma$ that contain a cell from~$\mathfrak s\cap\sigma$.
Clearly, we have~$U\cong\mathfrak s\times I$.

Since~$U$ is a ``thickened~$\mathfrak s$'', proving that~$K\times I$ collapses to~$U$ after a subdivision
amounts to showing that~$X$ collapses to~$\mathfrak s$, which we will demonstrate now.

One can see from Figure~\ref{blocks-fig} that~$X_-$ collapses to
$$\widetilde{\mathfrak s}\cup\left(\left[0;\frac12\right]\times\widetilde T_1\times\left[\frac12;\frac34\right]/{\sim}\right).$$
Indeed, this subcomplex is obtained from~$X_-$ by collapsing all cells in~$K\times0$ across the attached
cells (in descending order of dimension).

Similarly, $X_+$ collapses to
$$\widetilde{\mathfrak s}\cup\left(\left[0;\frac12\right]\times\widetilde T_2\times\left[\frac14;\frac12\right]/{\sim}\right),$$
which implies that~$X$ collapses to
$$\widetilde{\mathfrak s}\cup\left(\left(\left[0;\frac12\right]\times\widetilde T_1\times\left[\frac12;\frac34\right]\right)
\cup\left(\left[0;\frac12\right]\times\widetilde T_2\times\left[\frac14;\frac12\right]\right)/{\sim}\right).$$
The latter is identical to
\begin{equation}\label{kv12m..-eq}
\left((K\setminus V)\times\frac12\right)\cup\widetilde M\cup
\left(\left(\left[0;\frac12\right]\times\widetilde T_1\times\left[\frac12;\frac34\right]\right)
\cup\left(\left[0;\frac12\right]\times\widetilde T_2\times\left[\frac14;\frac12\right]\right)/{\sim}\right).
\end{equation}

Let~$p$ be a vertex of~$\widetilde T_1$ (for instance, a point in~$J$). Since~$\widetilde T_1$
is a tree, we have~$\widetilde T_1\searrow p$, which implies that
$$\left(\left[0;\frac12\right]\times\widetilde T_1\times\left[\frac12;\frac34\right]\right)/{\sim}$$
collapses to
$$\left(\left[0;\frac12\right]\times\widetilde T_1\times\left\{\frac12,\frac34\right\}\right)\cup
\left(\frac12\times\widetilde T_1\times\left[\frac12;\frac34\right]\right)\cup
\left(\left[0;\frac12\right]\times p\times\left[\frac12;\frac34\right]\right)/{\sim},$$
which, in turn, collapses to
$$\left(\left[0;\frac12\right]\times\widetilde T_1\times\left\{\frac12,\frac34\right\}\right)\cup
\left(\frac12\times\widetilde T_1\times\left[\frac12;\frac34\right]\right)/{\sim}=
\widetilde M\cap\left(\left(\left[0;\frac12\right]\times\widetilde T_1\times\left[\frac12;\frac34\right]\right)/{\sim}\right).$$

Similarly, 
$$\left(\left[0;\frac12\right]\times\widetilde T_2\times\left[\frac14;\frac12\right]\right)/{\sim}$$
collapses to
$$\widetilde M\cap\left(\left(\left[0;\frac12\right]\times\widetilde T_2\times\left[\frac14;\frac12\right]\right)/{\sim}\right).$$
This implies that the complex~\eqref{kv12m..-eq} collapses to
$$\left((K\setminus V)\times\frac12\right)\cup\widetilde M,$$
and so does~$X$. Since~$\widetilde T_i\searrow T_i$ for~$i=1,2$, it follows that~$\widetilde M\searrow M$,
which yields
$$X\searrow\left((K\setminus V)\times\frac12\right)\cup M=\mathfrak s,$$
and this completes the proof.
\end{proof}

\subsection{Matveev moves}

In~\cite{mat87,mat87+}, Matveev deduces ZCC restricted to standard polyhedra not embeddable into a three-manifold from
ACCS by means of several transformations of standard polyhedra, which are referred to here as \emph{Matveev moves}.
These moves are precisely the ones denoted in~\cite{mat87} by~$\slbf T_1$, $\slbf T_1^{-1}$, $\slbf T_2$, $\slbf T_2^{-1}$,
and~$\slbf T_3^{-1}$ (the move~$\slbf T_3$ is omitted). The reader is referred to~\cite{mat87} for the definitions of the moves. They are also described
in~\cite{mat07},
where $\slbf T_1$-, $\slbf T_2$-, and $\slbf T_3$-moves are called $\slbf V$-, $\slbf T$-, and $\slbf U$-moves,
respectively (see \cite[Definitions~1.2.6, 1.2.3, and~1.3.7]{mat07}). The $\slbf T_2$-move
is also known as the dual of the 2--3 Pachner move. Alternative definitions of the moves are given
in the proof of Lemma~\ref{matmov-lem} below.

To establish his result, Matveev proves two main points: first, if two standard polyhedra~$K$ and~$K'$
are three-deformable to each other and~$K$ is not embeddable into a three-manifold,
then~$K'$ can be obtained from~$K$ by a sequence of these moves; and second, these moves
are, in our terminology, $1$-admissible ($\slbf T_1$- and~$\slbf T_2$-moves are even admissible).

\begin{rema}
Matveev's terminology and exposition are quite different from ours. In particular, he discusses
the transformation of~$K'$ into~$K$, and accordingly uses $\slbf T_3$-moves rather than~$\slbf T_3^{-1}$.
Furthermore, he proves the $1$-admissibility of his moves in a somewhat stronger sense.
\end{rema}

We will apply some of Matveev's techniques in the final part of the proof of Theorem~\ref{K->semi-standard-th}.
For this reason, we briefly indicate how Matveev moves are related to the transformations considered above.

\begin{lemm}\label{matmov-lem}
Any Matveev move can be decomposed into a sequence of contractions and special extractions.
\end{lemm}

\begin{proof}
We call a true vertex of a polyhedron \emph{standard} if its link is homeomorphic to~$\mathscr K_4$.
By a \emph{bigon} (\emph{triangle}) in a polyhedron~$K$ we mean a two-disc~$d\subset K$ such
that its frontier~$\dot d$ coincides with~$d\cap K^{(1)\mathrm t}$ and contains exactly two (respectively, three) true vertices
of~$K$. A bigon (or triangle)~$d$ is called \emph{standard} if all the true vertices in~$\dot d$ are standard.
It is called \emph{twisted} if its regular neighborhood is not embeddable into~$\mathbb R^3$,
and \emph{untwisted} otherwise.

One can verify that Matveev $\slbf T_1$-moves are precisely $T_1,T_2$-extractions at a standard vertex,
where~$T_1$ and~$T_2$ are regular neighborhoods of two opposite edges of~$\mathscr K_4$.
Accordingly, each Matveev $\slbf T_1^{-1}$-move is a contraction, namely the contraction of the union of
two untwisted standard bigons.

The Matveev $\slbf T_2$-move can be described as follows: $K\mapsto K'$ is a $\slbf T_2$-move if
there exist a true edge~$e$ of~$K$ connecting two distinct standard vertices, and an untwisted standard
triangle~$t$ of~$K'$ such that the pairs~$(\overline{K\setminus N},\partial N)$
and~$(\overline{K'\setminus N'},\partial N')$ are homeomorphic, where~$N$ is a regular neighborhood of~$e$ in~$K$
and~$N'$ is a regular neighborhood of~$t$ in~$K'$.

In this case, $\partial N=\partial N'$ is the $1$-skeleton of a triangular prism. We call the edges contained
in the bases of the prism \emph{horizontal}, and the other three edges \emph{vertical}.

The transition from~$K'$ to~$K$ (which is a $\slbf T_2^{-1}$-move) can be decomposed into the following
two operations: the contraction of~$t$ followed by a $T_1,T_2$-extraction at the new vertex, where~$T_1$, $T_2$ are chosen
so that~$\partial T_i$
is the set of midpoints of the vertical edges of~$\partial N=\partial N'$.

The inverse transition (which is a $\slbf T_2$-move) can be decomposed into four operations as follows.
First, contract~$e$, then perform a $T_1,T_2$-extraction at the resulting vertex, where~$T_1$
is a regular neighborhood of a vertical edge. This creates two untwisted bigons, one of which is standard, while
the other is non-standard. Contract the non-standard bigon. This creates a true vertex whose link
is again the $1$-skeleton of a triangular prism. Perform a $T_1',T_2'$-extraction at this vertex, where~$T_1'$, $T_2'$ are
such that~$\partial T_i'$ is the set of midpoints of the vertical edges. The obtained
polyhedron is isomorphic to~$K'$.

Finally, the Matveev $\slbf T_3^{-1}$-move can be defined as the contraction of a twisted standard bigon. One can
verify that this definition is equivalent to Matveev's original one.
\end{proof}

\section{Standardization. Proof of Theorem~\ref{K->semi-standard-th}}\label{standardtn-sec}

For any two polyhedra~$K_1$ and~$K_2$, we have~$(K_1\#K_2)\times I\searrow(K_1\times[0;1/2])\cup(K_2\times[1/2;1])=
(K_1\times I)\#(K_2\times I)$, which implies the following statement.

\begin{lemm}\label{r-collapse-lem}
If polyhedra~$K_1,K_2,\ldots,K_m$ are $r$-collapsible, then any polyhedron of the form $K_1\# K_2\#\ldots\#K_m$
is $r$-collapsible.
\end{lemm}

Thus, to prove Theorem~\ref{K->semi-standard-th}, it remains to show that any compact contractible
two-polyhedron can be transformed into a semi-standard one by a sequence of admissible transformations.
As such transformations, we will use elementary collapses, disc contractions, and special extractions.
We have shown above that all of these transformations are admissible.

One can also readily see the following.

\begin{lemm}\label{K1K1'-lem}
Let~$K=K_1\#K_2$ be a polyhedron of dimension at most two. Suppose that~$K_1\mapsto K_1'$ is an
elementary collapse, a disc contraction, or a special extraction. Then there exists a polyhedron~$K'=K_1'\#K_2$
such that~$K\mapsto K'$ is an elementary collapse, a disc contraction, or a special extraction.
\emph(Note, however, that if~$K_1\mapsto K_1'$ is an elementary collapse, then~$K\mapsto K'$ may sometimes
be a disc contraction.\emph)
\end{lemm}

The proof of Theorem~\ref{K->semi-standard-th} is by induction on the ``non-standardness''
of the polyhedron~$K$. Let~$\mathscr C$ be the set of all compact contractible
two-polyhedra~$K$ (viewed up to homeomorphism) such that~$\partial K=\varnothing$ and~$K$
has no separating vertex. By Lemmas~\ref{r-collapse-lem} and~\ref{K1K1'-lem}, together with the fact
that any polyhedron with a non-empty boundary admits an elementary collapse, it suffices to show that
any~$K\in\mathscr C$ can be transformed into a connected sum of ``more standard'' polyhedra from~$\mathscr C$
by a sequence of elementary collapses, disc contractions, and special extractions.

We now define the non-standardness~$\nu(K)$ of a two-polyhedron~$K$.
For a true edge~$e$ of~$K$, we denote its degree by~$\deg(e)$. Let~$e_1,e_2,\ldots,e_l$
be all the true edges of~$K$ of degree at least four, numbered so that
$$\deg(e_1)\geqslant\deg(e_2)\geqslant\ldots\geqslant\deg(e_l).$$
We define the \emph{edge non-standardness~$\nu_{\mathrm e}(K)$} of~$K$ as the sequence~$(\deg(e_i))_{i=1,2,\ldots,l}$
and use the standard lexicographic order to compare these sequences.
Observe that the connected components of~$K^{(1)\mathrm t}$ that are homeomorphic to a circle
do not contribute to the edge non-standardness.

For a true vertex~$v$ of~$K$, we denote by~$\beta(v)$ the first Betti number of~$\mathrm{lk}(v,K)$.
Let~$v_1,v_2,\ldots,v_m$ be all the true vertices of~$K$ whose link is not homeomorphic to~$\mathscr K_4$,
numbered so that
$$\beta(v_1)\geqslant\beta(v_2)\geqslant\ldots\geqslant\beta(v_m).$$
We define the \emph{vertex non-standardness~$\nu_{\mathrm v}(K)$} of~$K$ as the sequence~$(\beta(v_i))_{i=1,2,\ldots,m}$
and again use the lexicographic order to compare these sequences.

We define the \emph{non-standardness} of~$K$ as the quadruple
$$\nu(K)=\bigl(\nu_{\mathrm e}(K),\nu_{\mathrm v}(K),\slbf v(K),b_1(K\setminus K^{(1)\mathrm t})\bigr),$$
where~$\slbf v(K)$ is the number of true vertices of~$K$ if~$K$ is not standard, and zero otherwise.
These quadruples form a well-ordered set with respect to the lexicographic order.

Clearly, a polyhedron $K\in\mathscr C$ is standard if and only if~$\nu(K)=(\varnothing,\varnothing,0,0)$.
The following lemma completes the proof of Theorem~\ref{K->semi-standard-th}.

\begin{lemm}\label{main-lem}
Suppose that~$K\in\mathscr C$ is not standard. Then~$K$ can be transformed by a sequence
of elementary collapses, disc contractions, and special extractions
into a polyhedron~$K'=K_1\#\ldots\#K_m$ with~$K_i\in\mathscr C$ and $\nu(K_i)<\nu(K)$
for~$i=1,\ldots,m$.
\end{lemm}

\begin{proof}
By the assumption of the lemma, one of the Conditions~(K4)--(K7) of Lemma~\ref{cases-k-lem} must hold.
We now consider them one by one.

\smallskip\noindent\emph{Case}~(K4): $K$ has a true edge of degree at least four.

Let~$e$ be the closure of a true edge of~$K$ of the maximum degree, and let~$d=\deg(e)$.

Suppose first that~$e$ connects two distinct
vertices of~$K$. Let~$K'$ be the polyhedron obtained from~$K$ by contracting~$e$.
We clearly have~$K'\in\mathscr C$ and~$\nu_{\mathrm e}(K')<\nu_{\mathrm e}(K)$, which implies that~$\nu(K')<\nu(K)$, and
this completes the subcase.

Suppose now, on the contrary, that~$e$ forms a loop based at a vertex~$v$ of~$K$.
We use the notation from Definition~\ref{resol-def} and identify~$G$ with~$\partial V$.
The edge~$e$ meets~$G$ at two vertices of~$G$, which we denote by~$u_1$ and~$u_2$.

Suppose that~$u_1$ is a separating vertex of~$G$. This implies that~$G$ is of the form~$G'\cup G''$,
where~$G'$ and~$G''$ are subgraphs such that~$G'\cap G''=\{u_1\}$. We may assume, without
loss of generality, that~$u_2\in G''$.

Let~$T_1$ be a regular neighborhood of~$u_1$ in~$G''$, and let~$T_2$ be any tree contained
in~$\overline{G''\setminus T_1}$ such that~$\partial T_2=\partial T_1$. If we set~$G_1=G'\cup T_1$
and~$G_2=\overline{G''\setminus T_1}$, then the graphs~$G_i$ and trees~$T_i$ satisfy all the properties
required in Definition~\ref{resol-def}.

We denote the apexes of the cones over~$G_1\cup T_2$, $T_1\cup T_2$, and~$T_1\cup G_2$
used in the construction of~$\mathfrak s(K,T_1,T_2)$ by~$a_1$, $a_2$, and~$a_3$, respectively.
We view them as vertices of~$\mathfrak s(K,T_1,T_2)$, although~$a_2$ and, optionally, $a_3$
might not be true vertices. Every true edge of~$\mathfrak s(K,T_1,T_2)$ either has a non-empty
intersection with~$K\setminus V$ or connects~$a_2$ with either~$a_1$ or~$a_3$.
In the former case, there exists a unique edge of~$K$ that contains this intersection and has the same degree.
In the latter case, the degree of the edge is equal to the degree of some vertex in~$T_1\cup T_2$,
which is strictly less than~$d$.

Let~$e'$ be the edge of~$\mathfrak s(K,T_1,T_2)$ containing~$e\setminus V$. This edge does not form a loop,
since it connects~$a_1$ with~$a_3$. Let~$K'''$ be the polyhedron
obtained from~$\mathfrak s(K,T_1,T_2)$ by contracting the edge~$e'$.
The sequence~$\nu_{\mathrm e}(K''')$ is obtained from~$\nu_{\mathrm e}(K)$ by removing one element equal
to~$d$ and adding a finite number of strictly smaller elements. Therefore,
$\nu_{\mathrm e}(K''')<\nu_{\mathrm e}(K)$, which implies that~$\nu(K''')<\nu(K)$.

The polyhedron~$K'''$ might not belong to~$\mathscr C$, as we may have~$\partial K'''\ne\varnothing$.
This occurs if and only if~$T_1=T_2$ is a single-point set, which requires~$G$ to have a separating edge.

Let~$K''$
be a minimal subpolyhedron of~$K'''$ such that~$K'''\searrow K''$, and let~$K'$ be the polyhedron obtained from~$K''$
by contracting each edge of zero degree (such an edge cannot form a loop as~$K''$ is contractible).

Then~$K'$ is of the form~$K_1\#K_2\#\ldots\#K_m$, where each~$K_i\in\mathscr C$. Moreover, we have~$\nu_{\mathrm e}(K_i)\leqslant
\nu_{\mathrm e}(K')\leqslant\nu_{\mathrm e}(K'')\leqslant\nu_{\mathrm e}(K''')<\nu_{\mathrm e}(K)$, which implies~$\nu(K_i)<\nu(K)$ for $i=1,\ldots,m$.
The polyhedron~$K'$ is obtained from~$K$ by a special extraction and a finite number of elementary collapses and edge
contractions, so the claim follows.

It remains to consider the subcase where neither~$u_1$ nor~$u_2$ is a separating vertex of~$G$ (if~$u_2$
is a separating vertex, we exchange~$u_1$ and~$u_2$ and proceed as above).
Let~$T_1=G_1$ be a small regular neighborhood of~$u_1$ in~$G$. The boundary~$\partial T_1$
contains exactly~$d$ points. We set~$G_2=\overline{G\setminus T_1}$.

Suppose that there exists a choice for~$T_2$ such that~$T_2$ does not contain a vertex of degree~$d$,
and thus has only vertices of strictly smaller degrees. If so, we make such a choice.

Let~$a_i$ for~$i=1,2,3$ and~$e'$ be as in the previous subcase.
Since~$T_1$ is homeomorphic to the star graph~$\mathscr S_d$, the vertices~$a_2$ and~$a_3$
are connected by a single edge of~$\mathfrak s(K,T_1,T_2)$ of degree~$d$. We denote this edge by~$e''$.
The vertices~$a_1$ and~$a_2$ are connected by edges whose degrees are strictly less than~$d$.

Let~$K'$ be the polyhedron
obtained from~$\mathfrak s(K,T_1,T_2)$ by contracting the edges~$e'$ and~$e''$. We then
have~$\nu_{\mathrm e}(K')<\nu_{\mathrm e}(K)$,
and~$K'\in\mathscr C$.

However, there may be no choice of~$T_2$ without a degree-$d$ vertex. This occurs only if there exists
a degree-$d$ vertex~$w$ of~$G_2$ such that any path in~$G_2$ connecting two distinct leaves of~$G_2$ passes through~$w$.
(This vertex may or may not coincide with~$u_2$.) In this case, $G\setminus\{u_1,w\}$ consists
of exactly~$d$ connected components. These components cannot all be open arcs since~$v$
is a true vertex of~$K$.

Let~$U$ be a regular neighborhood of~$\{u_1,w\}$. We redefine the subgraphs~$T_i$ and~$G_i$  of~$G$ for~$i=1,2$
as follows. Let~$G_1$ be a connected component of~$\overline{G\setminus U}$
that is not homeomorphic to an arc. Then the boundary of~$G_1$ consists of exactly two points.
We set~$G_2=\overline{G\setminus G_1}$, and let~$T_1$ and~$T_2$ be any two arcs in~$G_1$ and~$G_2$,
respectively, connecting the two points in~$\partial G_1=\partial G_2$. Finally, we set~$K'=\mathfrak s(K,T_1,T_2)$.
We then have~$K'\in\mathscr C$.

Let~$a_1$, $a_2$, and~$a_3$ be as before. Since~$T_1\cup T_2$ is a circle, the vertex~$a_2$ is not a true
vertex of~$K'$ (and it does not even belong to~${K'}^{(1)\mathrm t}$).
This implies, in particular, that~$K'$ contains no new true edges.
The vertex~$a_3$ may or may not be a true vertex of~$K'$. If it is, then we have~$\nu_{\mathrm e}(K')=\nu_{\mathrm e}(K)$,
whereas otherwise~$\nu_{\mathrm e}(K')<\nu_{\mathrm e}(K)$, since the number of true edges of degree $d$
decreases by one. In both cases, $\nu_{\mathrm v}(K')<\nu_{\mathrm v}(K)$, hence~$\nu(K')<\nu(K)$,
and this completes Case~(K4).

\smallskip\noindent\emph{Case}~(K5): there exists a vertex~$v$ of~$K$ such that the graph~$\mathrm{lk}(v,K)$
contains two disjoint cycles.

We assume that all the true edges of~$K$ have degree three (otherwise, we proceed as in Case~(K4)), thus we have~$\nu_{\mathrm e}(K)=\varnothing$.

We choose~$G_1,G_2\subset G=\mathrm{lk}(v,K)$ and~$T_1,T_2$ as in Definition~\ref{resol-def} in such a way that~$G_1$ and~$G_2$
are not simply connected, and let~$K'''$ be the polyhedron obtained from~$K$ by the $T_1,T_2$-extraction at~$v$. One can see
that no true edge of degree strictly greater than three can arise; thus we have~$\nu_{\mathrm e}(K''')=\varnothing$,
whereas the vertex non-standardness decreases: $\nu_{\mathrm v}(K''')<\nu_{\mathrm v}(K)$.

If~$T_1$ and~$T_2$ consist of more than one point, then~$\partial K'''=\varnothing$ and~$K'''\in\mathscr C$. In this case,
we set~$K'=K'''$. Otherwise, we let~$K''$ be a minimal subpolyhedron of~$K'''$ such that~$K'''\searrow K''$,
and let~$K'$ be the polyhedron obtained from~$K''$ by contracting every zero-degree edge.
In both cases, $K'$ is of the form~$K_1\#\ldots\#K_m$ with~$K_i\in\mathscr C$, $\nu_{\mathrm e}(K_i)=\varnothing$, and~$\nu_{\mathrm v}(K_i)<
\nu_{\mathrm v}(K)$ for~$i=1,\ldots,m$.

\smallskip
\noindent\emph{Case} (K6): there exists a vertex~$v$ of~$K$ such that~$\mathrm{lk}(v,K)$ is
homeomorphic to~$\mathscr K_{3,3}$.

Let~$G_1=T_1\subset G=\mathrm{lk}(v,K)$ be a regular neighborhood of an edge of~$G$,
and let~$T_2\subset G_2=\overline{G\setminus G_1}$ be a tree such that~$T_1\cup T_2\cong\mathscr K_4$;
see Figure~\ref{k33-t2-fig}, where~$T_2$ is shown as a bold line.

The $T_1,T_2$-extraction at~$v$ yields a polyhedron~$K'''=\mathfrak s(K,T_1,T_2)$ whose
three new vertices~$a_1$, $a_2$, and~$a_3$ have links homeomorphic to~$\mathscr K_4$, $\mathscr K_4$,
and~$\mathscr K_{3,3}$, respectively; see Figure~\ref{k33aaa-fig}. The dashed lines
connect the vertices of these graphs that correspond to the same true edges of~$K'''$.
\begin{figure}[ht]
\includegraphics[scale=.6]{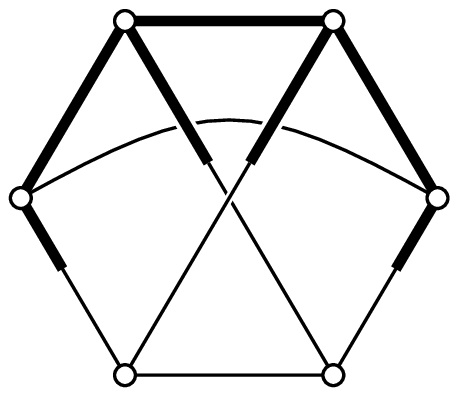}
\caption{The tree~$T_2$ in the case where~$G=\mathrm{lk}(v,K)\cong\mathscr K_{3,3}$}\label{k33-t2-fig}
\end{figure}

One can see that there exists a bigon whose boundary contains~$a_2$ and~$a_3$.
Let~$K''$ be the polyhedron obtained by contracting this bigon.
Then~$K''$ is obtained from~$K$ by replacing a regular neighborhood of~$v$
with a complex having two true vertices, one of which has a link homeomorphic to~$\mathscr K_4$,
while the other has a link homeomorphic to the $1$-skeleton of a triangular prism. We have~$\nu(K'')=\nu(K)$, but
now~$K''$ has a vertex whose link contains two disjoint cycles. So, we proceed
as in Case~(K5) to obtain a polyhedron~$K'$ that has a strictly smaller non-standardness than~$K$.
\begin{figure}[ht]
\includegraphics[scale=.5]{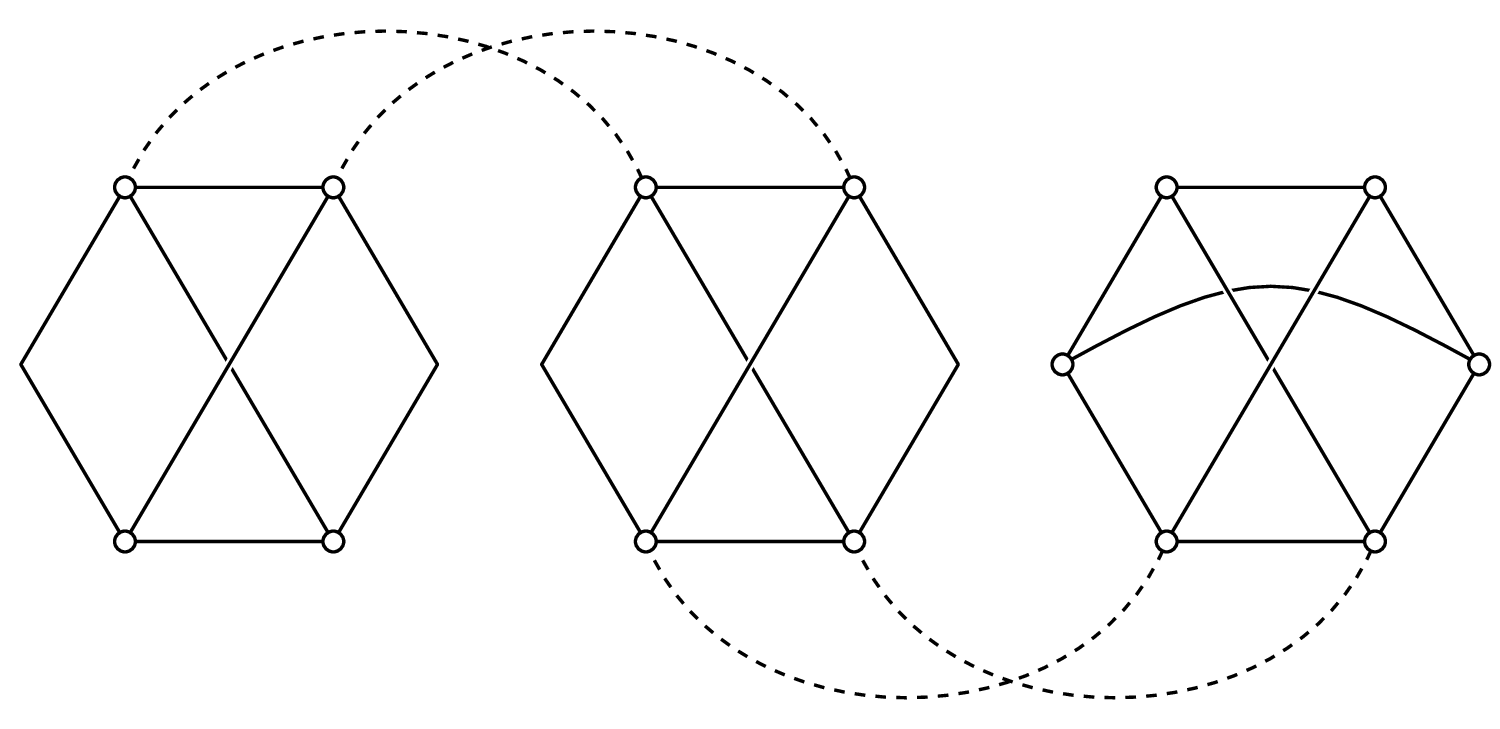}\put(-325,30){$G_1\cup T_2$}\put(-193,30){$T_1\cup T_2$}\put(-80,30){$T_1\cup G_2$}
\caption{The links of the vertices~$a_1$, $a_2$, and~$a_3$ after the $T_1,T_2$-extraction in the case where~$G=\mathrm{lk}(v,K)\cong\mathscr K_{3,3}$}
\label{k33aaa-fig}
\end{figure}

\noindent\emph{Case}~(K7):
there exists a connected component of~$K\setminus K^{(1)\mathrm t}$ that is not simply connected.
A~homotopically non-trivial simple closed curve in~$K\setminus K^{(1)\mathrm t}$ is called \emph{essential}.

Let~$P\ne\varnothing$ be a minimal contractible subpolyhedron of~$K$ such that~$\partial\mathscr N(P)$ is a disjoint union (possibly empty) of essential simple closed curves. We claim that each connected component of~$P\setminus K^{(1)\mathrm t}$ is a two-disc (which implies, in particular,
that~$P$ is a proper subset of~$K$ and~$\partial\mathscr N(P)\ne\varnothing$).

Indeed, suppose otherwise.
Let~$\alpha$ be an essential simple closed curve in~$P\setminus K^{(1)\mathrm t}$.
The regular neighborhood~$\mathscr N(\alpha)$ cannot be a M\"obius band, since otherwise~$\alpha$
would be homologically non-trivial in~$P$. Hence, $\alpha$ is two-sided, meaning that~$\mathscr N(\alpha)$ is an annulus.

The complement~$P\setminus\alpha$ is disconnected, since otherwise there
would exist a simple closed curve in~$P$ that intersects~$\alpha$ transversally
at a single point. Such a curve would be homotopically non-trivial in~$P$.

Let~$W_1$ and~$W_2$ be the two connected components of~$P\setminus\alpha$, and let~$P_i=\overline W_i$ for~$i=1,2$.
We claim that one of the subpolyhedra~$P_1$ and~$P_2$ is contractible.
Indeed, the homotopy classes of~$\alpha$ in~$\pi_1(P_1)$ and~$\pi_1(P_2)$
cannot both be of finite order, for otherwise~$H_2(P)$ would be non-trivial.
Without loss of generality, we may assume that~$\alpha$ represents
an element of~$\pi_1(P_2)$ of infinite order. Then,
since~$\pi_1(P)=\{1\}$, the loop~$\alpha$ is homotopically trivial in~$P_1$.

This implies that~$P$ is homotopy equivalent to a connected sum~$P_1\#(P_2/\alpha)$;
hence, $P_1$ is contractible. By construction,  $\partial\mathscr N(P_1)$
is a disjoint union of essential simple closed curves, and~$P_1$ is a proper subset of~$P$, which contradicts the minimality of~$P$.

The polyhedron $P$ might contain no true vertices of~$K$,
which means that~$P$ is a two-disc with~$P\cap K^{(1)\mathrm t}=\partial P$. The contraction of~$P$ turns~$K$ into a connected sum
of polyhedra~$K_1,\ldots,K_m\in\mathscr C$, where~$m$ is the number of connected components of~$\partial\mathscr N(P)$.
For every~$i=1,\ldots,m$, we have~$\nu_{\mathrm e}(K_i)=\nu_{\mathrm v}(K_i)=\varnothing$,
$\slbf v(K_i)\leqslant\slbf v(K)$, and~$b_1(K_i\setminus K_i^{(1)\mathrm t})<b_1(K\setminus K^{(1)\mathrm t})$;
thus, the claim follows in this case.

In what follows, we assume that~$P$ contains a true vertex of~$K$. Let~$Q$ be a regular neighborhood of~$P$ in~$K$,
and let~$A$ be a connected component of~$Q\setminus P$. If~$Q$ is embeddable into a three-manifold,
then we choose an embedding~$Q\hookrightarrow I^3$ such that~$Q\cap\partial I^3=\partial Q$
(it exists due to the Poincar\'e--Perelman theorem),
and let~$A$ be a connected component of~$Q\setminus P$ containing an outermost component of~$\partial Q$ in~$\mathbb S^2=
\partial I^3$.

By construction, we have~$A\cong\mathbb S^1\times[0;1)$.
Let~$\widetilde K$ be the connected component of~$K\setminus A$ that contains~$P$, and let~$\widehat K$
be the polyhedron obtained from~$K$ by contracting~$\widetilde K$ to a point. We claim that~$K$ can be transformed
into~$\widetilde K\#\widehat K$ by a sequence of elementary collapses, contractions, and special extractions.

Before proving this claim, we show that it implies the assertion of the lemma. Indeed, both~$\widetilde K$ and~$\widehat K$
have fewer true vertices than~$K$ and have empty edge and vertex non-standardnesses.
The polyhedron~$\widehat K$ already belongs to~$\mathscr C$, whereas~$\widetilde K$ may have a non-empty boundary.
By collapsing~$\widetilde K$ as much as possible and contracting all zero-degree edges, one obtains a connected sum of polyhedra from~$\mathscr C$
that have strictly smaller non-standardness than~$K$.

It remains to explain how to transform~$K$ into~$\widetilde K\#\widehat K$.

Since~$P$ is not a disc, there exists a true vertex~$v$ of~$K$ in~$\overline A$.
Let~$d=\mathscr N(v)\cap\overline A$. Clearly, $d$ is a two-disc whose boundary consists
of two arcs~$\alpha=d\cap P\subset K^{(1)\mathrm t}$ and~$\beta=\overline{\partial d\cap A}$. Let~$\beta'$
be an arc in~$\overline A$ such that~$\partial\beta'=\partial\beta$, $\beta'\setminus\partial\beta'\subset A$,
and~$\beta\cup\beta'$ forms a simple closed curve isotopic to the core circle of~$A$ (consult Figure~\ref{alpha-beta-gamma-fig}).
\begin{figure}[ht]
\includegraphics[scale=.75]{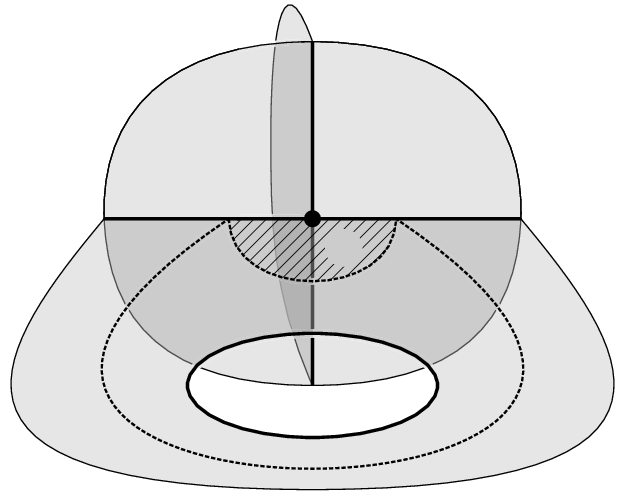}\put(-100,105){$\alpha$}
\put(-100,72){$\beta$}\put(-35,40){$\beta'$}\put(-110,130){$\delta$}
\put(-105,87){$d$}\put(-220,10){$A$}\put(-115,28){$\gamma$}
\put(-122,105){$v$}\put(-137,167){$P$}
\hskip5mm
\includegraphics[scale=.75]{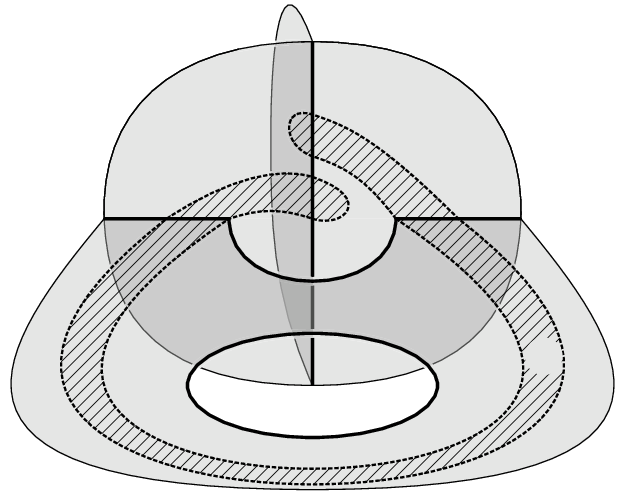}\put(-20,40){$\alpha'$}
\put(-33,48){$d'$}\put(-77,120){$\alpha'$}\put(-150,116){$\alpha'$}
\put(-124,-13){$\overline{Q\setminus d}$}
\caption{Constructing~$Q_1$ from~$Q$}\label{alpha-beta-gamma-fig}
\end{figure}
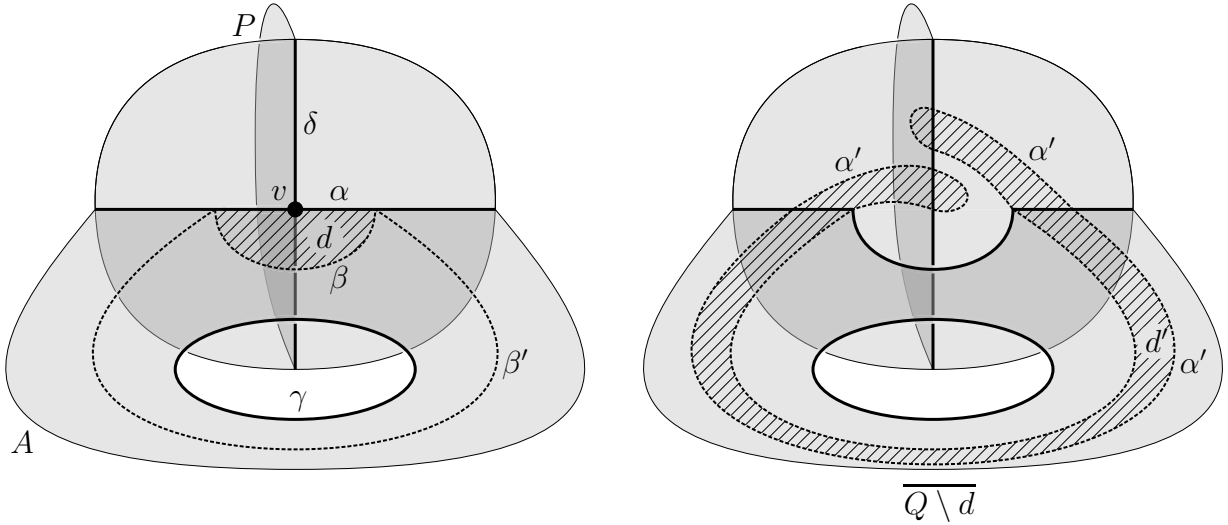
Furthermore, let~$\gamma$ be the component of~$\partial Q$ contained in~$A$, and let~$\delta$ be the arc $\overline{\mathscr N(v)\cap
K^{(1)\mathrm t}\setminus\alpha}$.

Suppose that~$Q$ is not embeddable into a three-manifold. One can find a two-disc~$d'$ contained
in a small neighborhood of~$\alpha\cup\beta'$ whose boundary consists of two arcs, one of which is~$\beta'$
while the other, which we denote by~$\alpha'$, has the following properties:
\begin{enumerate}
\item
$\partial\alpha'=\partial\alpha=\alpha'\cap\beta$;
\item
$\alpha'$ meets~$\overline{K^{(1)\mathrm t}\setminus\alpha}$ at exactly eight points, two of which are the endpoints, and four of which lie on~$\delta$;
\item
the points of~$\alpha'\cap\overline{K^{(1)\mathrm t}\setminus\alpha}$ appear on~$\alpha'$ in the following order:
an endpoint, two points of~$\alpha'\cap\delta$, two points of~$\alpha'\cap K^{(1)\mathrm t}\setminus(\partial\alpha'\cup\delta)$,
another two points of~$\alpha'\cap\delta$, and the other endpoint;
\item
the portion of~$\alpha'$ between the two points of~$\alpha'\cap K^{(1)\mathrm t}\setminus(\partial\alpha'\cup\delta)$
is contained in~$A$ and runs ``parallel'' to~$\beta'$.
\end{enumerate}

Let~$Q_1$ and~$K_1$ be the polyhedra obtained from~$Q$ and~$K$, respectively, by removing the interior of~$d$
and attaching a two-disc along the circle~$\alpha'\cup\beta$. The arc~$\alpha$ is homotopic
in~$\overline{Q\setminus d}$ to~$\alpha'$ relative to $\partial\alpha$. Using Matveev's technique, one can deduce from this that
the transition from~$Q$ to~$Q_1$ can be performed by a sequence of
Matveev moves. It is explained in~\cite[\S3]{mat87} how to produce a sequence of~$\slbf T_i^{\pm1}$-moves, $i=1,2,3$,
and the method of excluding $\slbf T_3$-moves is described in~\cite[\S5]{mat87}.

Since the polyhedra~$Q$ and~$Q_1$
are not standard, one cannot apply the results of~\cite{mat87} directly; however one can see that the
procedure described in~\cite{mat87} works perfectly well in the present situation. This is due to the fact that~$Q$ and~$Q_1$
become standard after attaching a two-disc to each connected component of~$\partial Q=\partial Q_1$. Then the procedure
described in~\cite{mat87} can be performed for these polyhedra without involving the attached discs.

To see this, two main facts must be taken into account. First, $\alpha$ and~$\alpha'$ are already homotopic in~$\overline{Q\setminus d}$. Second,
since~$Q$ is not embeddable into a three-manifold, there exists a connected component~$\sigma$ of~$P\setminus
K^{(1)\mathrm t}$ such that~$\mathscr N(\overline\sigma)$ is also not embeddable into a three-manifold. Both of
these facts are due to the contractibility of~$P$.

Thus, by Lemma~\ref{matmov-lem}, the polyhedron~$K_1$ can be obtained from~$K$ by a sequence
of contractions and special extractions.

Now let~$K_2$ be the polyhedron obtained from~$K_1$ by contracting~$d'$ to a point. This contraction creates a new vertex~$w$.
The link of~$w$ in~$K_2$ consists of two connected components, one of which is a circle, while the other has a separating
edge. After the $(T_1,T_2)$-extraction with a single-point tree~$T_1=T_2$, followed by a single elementary collapse,
we obtain a polyhedron isomorphic to~$\widetilde K\#\widehat K$, which completes this subcase.

Now consider the situation where~$Q$ is embeddable into a three-manifold. The procedure in this case is essentially
the same. The only difference is that the disc~$d'$ must be chosen more carefully, and the argument why~$K$
can be transformed into~$K_1$ by a sequence of Matveev moves is slightly different (in particular, $\slbf T_3^{-1}$-moves
are not required here).

We view~$Q$ as a subpolyhedron of~$I^3$ such that~$\partial Q=Q\cap\partial I^3$. The three-disc~$I^3$
can be identified with the mapping cylinder of a projection map~$\pi$ from~$\mathbb S^2=\partial I^3$ to~$Q$.
The component~$\gamma$ of~$\partial Q$ is chosen to be outermost in~$\mathbb S^2=\partial I^3$.
This means that there exists a two-disc~$g\subset\mathbb S^2$ such that~$g\cap Q=\partial g=\gamma$.

The disc~$d'$ must be chosen so as to satisfy the additional condition that there exists a two-disc~$\widetilde d$ contained in~$g$
such that~$d'=\pi(\widetilde d)$. One can see that one of the several options for~$d'$ meets this restriction.
The homotopy from~$\alpha$ to~$\alpha'$ that is used to produce a sequence
of $\slbf T_i^{\pm1}$-moves should also be chosen so that it can be ``lifted'' to~$g$.
In this case, all the polyhedra that arise in these moves are embeddable into~$I^3$, so the~$\slbf T_3^{\pm1}$-moves
are not used. (One can even succeed without~$\slbf T_1^{\pm1}$-moves
by using the technique explained in~\cite[Section~1.2]{mat07}.)

This completes the proof of Lemma~\ref{main-lem} and Theorem~\ref{K->semi-standard-th}.
\end{proof}

\subsection*{Acknowledgements}
This work was performed at the Steklov International Mathematical Center and supported by the Ministry of Science and Higher Education of the Russian Federation (agreement no. 075-15-2025-303).

Google Gemini (Web version) was used solely for the purpose of final language editing.

\end{document}